\documentclass{amsart}
\usepackage{amsmath}
\usepackage{amssymb}
\usepackage{array}
\usepackage{amscd}
\usepackage{tikz-cd}
\usepackage{tikz}
\usepackage{tkz-graph}
\usepackage[all]{xy}
\usepackage{float}
\usepackage{subcaption}
\usetikzlibrary{calc}

\newcommand*{\DashedArrow}[1][]{\mathbin{\tikz [baseline=-0.25ex,-latex, dashed,#1] \draw [#1] (0pt,0.5ex) -- (1.3em,0.5ex);}}%

\numberwithin{equation}{section}
\newtheorem{Th}[subsection]{Theorem}
\newtheorem*{Th*}{Theorem}
\newtheorem{Lemma}[subsection]{Lemma}
\newtheorem{Prop}[subsection]{Proposition}
\newtheorem{Conjecture}[subsection]{Conjecture}
\newtheorem{Cor}[subsection]{Corollary}
\theoremstyle{definition}
\newtheorem{definition}[subsection]{Definition}
\newtheorem*{definition*}{Definition}
\newtheorem{Remark}[subsection]{Remark}
\newtheorem{Example}[subsection]{Example}
\newcommand{\comm}[1]{}

\let \P \relax
\newcommand{\P}{\mathbb{P}}
\newcommand{\F}{\mathbb{F}}
\newcommand{\Cl}{\mathcal{C}}
\newcommand{\Dl}{\mathcal{D}}

\usepackage{color}
\usepackage[normalem]{ulem}
\definecolor{DarkGreen}{rgb}{0,0.5,0.1} 

\newcommand\soutD{\bgroup\markoverwith
{\textcolor{DarkGreen}{\rule[.5ex]{2pt}{1pt}}}\ULon}

\definecolor{darkgreen}{RGB}{50,162,69}  

\makeatletter
\newcommand{\romsmall}[1]{\romannumeral #1}
\newcommand*{\rom}[1]{\expandafter\@slowromancap\romannumeral #1@}
\makeatother

\begin{document}
 
\title{Composition of Sarkisov links between del Pezzo surfaces}
\author{Anastasia V.~Vikulova}
\address{{\sloppy
\parbox{0.9\textwidth}{
Department of Mathematics and Computer Science, University of Basel, Spiegelgasse
1, 4051 Basel, Switzerland
}\bigskip}}
\email{vikulovaav@gmail.com}
\date{}
\thanks{}
\maketitle

\begin{abstract}
We prove that for any two birationally equivalent  del Pezzo surfaces of Picard rank~one over a perfect field there is a birational map between them such that it is decomposed in a composition of at most two Sarkisov links. 

\end{abstract}

\tableofcontents

\section{Introduction}

Recall that the Sarkisov program states that any birational map between two Mori fibre spaces can be decomposed in a composition of elementary transformations which are called Sarkisov links. More details about Sarkisov program over~$\mathbb{C}$ in arbitrary dimension can be found in~\cite{Hacon}. For Sarkisov program over~$\mathbb{C}$ in dimension $3$ we refer the readers to~\cite{CortiSarkisov}. For Sarkisov program over arbitrary perfect field in dimension $2$ we refer the readers to~\cite{Isk}.

In the paper~\cite{Corti} A.~Corti proposed the conjecture in the Sarkisov program in dimension $3.$ Concerning terminal Fano varieties of Picard rank one (which are examples of Mori fibre spaces of dimension $3$) it is anticipated the following. 

\begin{Conjecture}[{\cite[Conjecture 2.1]{Corti}}]\label{conj:Corti}

There exists a constant $N$ such that for any birationally equivalent terminal Fano  varieties $X$ and $X'$ of dimension $3$ of Picard rank one over $\mathbb{C}$ there exists a birational map $f  \colon X \DashedArrow[->,densely dashed    ] X'$ such that it is a composition of $n$ Sarkisov links, where $n \leqslant N.$

\end{Conjecture}

Unfortunately, we cannot expect that analogue of Conjecture~\ref{conj:Corti} holds in case of dimension~$4$ and higher. For this by~\cite{Corti}  it is enough to show that in dimension~$4$ there is a non-constructible subset of birational Fano varieties of Picard rank $1$ in the coarse moduli space. Consider a coarse moduli space $\mathcal{C}$ of smooth cubic hypersurfaces of dimension~$4$ over $\mathbb{C}.$ In~\cite[Theorem 6.8]{Katzarkov} it was proved that a very general cubic hypersurface of dimension~$4$ in $\mathbb{P}^5_{\mathbb{C}}$ is not rational. From Kuznetsov's conjecture (see~\cite[Conjecture 1.1]{Kuznetsov})  it follows that there are countably infinite many irreducible divisors  in $\mathcal{C}$ of rational  cubic hypersurfaces of dimension~$4$ (see~\cite{Hassett},~\mbox{\cite[\S 3.3]{Hassettlecture}},~\cite{Russo} for more details). So this subset is not constructible.

Our goal is to prove the analogue of Conjecture~\ref{conj:Corti} in case of dimension $2$ and arbitrary perfect field. 

\begin{Th}\label{th:Sarkisovperfect}
Let $X$ and $X'$ be smooth del Pezzo surfaces of Picard rank one over a perfect field and assume that $X$ is birational to $X'$. Then there exists a birational map $f  \colon X \DashedArrow[->,densely dashed    ] X'$ such that it is a composition of at most~\mbox{$N=2$} Sarkisov links. Moreover, this $N$ is optimal, that is there is a perfect field~$\mathbb{K}$ and two birationally equivalent  del Pezzo surfaces $X$ and $X'$  of Picard rank one over~$\mathbb{K}$ such that any birational map $f \colon X \DashedArrow[->,densely dashed    ] X'$ is a composition of at least two Sarkisov links.

\end{Th}

\begin{Remark}
Let us justify why we only consider del Pezzo surfaces among all Mori fibre spaces of dimension $2$ in Theorem~\ref{th:Sarkisovperfect}.  Consider Hirzebruch surfaces~$\mathbb{F}_n.$ It is well known that all rational Hirzebruch surfaces are birationally equivalent to each other. If we consider Hirzebruch surfaces over algebraically closed field, any birational map between $\mathbb{F}_n$ and $\mathbb{F}_m$ is a composition of at least $|m-n|$ Sarkisov links, because  Sarkisov links between Hirzebruch surfaces are 
\begin{equation}\label{eq:hirzebruch}
\mathbb{F}_n  \DashedArrow[->,densely dashed    ] \mathbb{F}_{n'},
\end{equation}
where $n'=n \pm 1.$

Let us consider non-algebraically closed field, for example, the field of real numbers $\mathbb{R}.$ In this case  we have  that  Sarkisov links between Hirzebruch surfaces are~\eqref{eq:hirzebruch}, where $|n-n'| \leqslant 2.$ So any birational map between   $\mathbb{F}_n$ and $\mathbb{F}_m$ is a composition of at least $\lfloor \frac{m-n}{2} \rfloor$ Sarkisov links. So in the case of conic bundles the analogue of Theorem~\ref{th:Sarkisovperfect} is false. Note, however, that over $\mathbb{Q}$ for any non-negative~$n$ and $m$ there is a birational map $f \colon \mathbb{F}_n  \DashedArrow[->,densely dashed    ] \mathbb{F}_{m},$ such that $f$ is a Sarkisov link, since for any positive integer $r$ there is a Galois orbit of length $r$ of a point on the Hirzebruch surfaces in general position.

Finally, note that Sarkisov program for del Pezzo surfaces of Picard rank $1$ over algebraically closed field in dimension~$2$ is trivial, since the only del Pezzo surface of Picard rank~$1$ is the projective plane~$\mathbb{P}^2.$

\end{Remark}

Now let us describe the structure of the paper.
In Section~\ref{section:lemmas} we give some general facts which we need and recall the definition of Sarkisov links.
In Section~\ref{section:dP6} we recall some basic facts about del Pezzo surfaces of degree $6.$ 
In Section~\ref{section:invariants} we study invariants of del Pezzo surfaces of high degree.
In Section~\ref{section:existencepointssarkisov} we study the existence of points in Sarkisov general position on minimal del Pezzo surfaces of degree~$5$ and~$6.$
In Section~\ref{section:existencesarkisov} we study  Sarkisov links of type~\rom{2} between minimal rational del Pezzo surfaces. 
In Section~\ref{section:sarkisov}  we estimate the minimal number of Sarkisov links in the composition of a birational map between minimal rational del Pezzo surfaces.
In Section~\ref{section:proof} we prove Theorem~\ref{th:Sarkisovperfect}.

\vspace{0.5cm}

\textbf{Notation.} We work over a perfect field~$\mathbb{K}.$ All varieties are smooth.  Let $X$ be a variety (or a scheme) defined over the field $\mathbb{K}.$ If $\mathbb{K} \subset \mathbf{L}$ is an extension of $\mathbb{K},$ then we will denote by $X_{\mathbf{L}}$ the variety (or the scheme)
$$
X_{\mathbf{L}}=X \times_{\mathrm{Spec}(\mathbb{K})}\mathrm{Spec}(\mathbf{L})
$$

\noindent over $\mathbf{L}.$    By $\overline{\mathbb{K}}$ we denote the algebraic closure of $\mathbb{K}.$ We say that a zero-dimensional scheme $P$ defined over $\mathbb{K}$ is a point of degree $d,$ if $P_{\overline{\mathbb{K}}}$ is~$\mathrm{Gal}(\overline{\mathbb{K}}/\mathbb{K})$-orbit of length~$d.$ We say that $P$ is a $\mathbb{K}$-point, if $P_{\overline{\mathbb{K}}}$ is a point.   Given a finite field extension~\mbox{$\mathbb{K} \subset \mathbf{L}$} and a scheme $Y$ over $\mathbf{L}$ denote by  $\mathrm{R}_{\mathbf{L}/\mathbb{K}}(Y)$  its Weil restriction of scalars. Recall that $\mathrm{R}_{\mathbf{L}/\mathbb{K}}$ is a functor from the category of schemes over $\mathbf{L}$ to the category of schemes over $\mathbb{K},$ which is right adjoint to the fiber product with~$\mathrm{Spec}(\mathbf{L}).$  By $\mathfrak{S}_n$ we denote the symmetric group of degree $n,$  by $\mathfrak{A}_n$ we denote the alternating group of degree $n$ and by $\mathfrak{D}_n$ we denote the dihedral group of order~$2n.$

\vspace{0.5cm}

\textbf{Acknowledgment.} The author is deeply grateful to  C.~A.~Shramov for his  intellectual generosity, constant support, guidance and belief in author's abilities.  Also the author is profoundly indebted to  A.~S.~Trepalin for constant support, constant opportunity to learn from him, for carefully reading this paper and helping to make it more readable, and for endless patience. The author is profoundly indebted to K.~V.~Kvitko and A.~A.~Kuznetsova for their constant support and friendship.

The author is supported by the project SERI REF-1131-552105.

\section{Preliminaries}\label{section:lemmas}

In this section we collect some auxiliary facts and recall general theory about Sarkisov links between del Pezzo surfaces. The first fact is the well-known theorem of Lang and Nishimura.

\begin{Th}[{see e.g.~\cite[Lemma~1.1]{VA}}]
\label{theorem:Lang-Nishimura}
Let $X$ and $Y$ be smooth projective varieties over a field $\mathbb{K}$.
Suppose that $X$ is birational to~$Y$. Then $X$ has a $\mathbb{K}$-point if and only if~$Y$ has a $\mathbb{K}$-point.
\end{Th}

It is well-known fact that every smooth del Pezzo surface of degree $5$ always has a $\mathbb{K}$-point.

\begin{Th}[{see, for instance,~\cite{Barron} or~\cite{Swin}}]\label{th:dp5point}
Let $X$ be a del Pezzo surface of degree $5$ over a field $\mathbb{K}.$ Then $X$ has a $\mathbb{K}$-point.
\end{Th}

Let us also recall well-known result about rationality of  del Pezzo surfaces of degree $d \geqslant 5$ with a $\mathbb{K}$-point.

\begin{Th}[{\cite[Theorem 2.1]{VA}}]\label{theorem:rationalityd>5}
Let $X$ be a del Pezzo surface of degree $d \geqslant 5$ over $\mathbb{K}.$ If $X$ has a $\mathbb{K}$-point,  then  $X$ is $\mathbb{K}$-birational to $\mathbb{P}^2.$

\end{Th}

\begin{definition}
 Let $X$ be a  del Pezzo surface and let $P \subset X$ be a collection of distinct points. We say that $P$ is \textit{in Sarkisov general position} if the blow up of~$X$ at $P$ is again a  del Pezzo surface.
\end{definition}

Now we give an important definition in the birational geometry of surfaces.

\begin{definition}
A smooth projective surface $X$ is called \textit{minimal} if any birational morphism $f \colon X \to X'$ to a smooth projective surface $X'$ is an isomorphism.
\end{definition}

\begin{Remark}
It is a classical fact (cf.~\cite[Theorem 1]{IskMMM}) that if $X$ is a minimal geometrically rational surface, then we have the following two mutually exclusive cases:

\begin{enumerate}
\renewcommand\labelenumi{\rm (\arabic{enumi})}
\renewcommand\theenumi{\rm (\arabic{enumi})}

\item
the Picard group of $X$ is $\mathrm{Pic}(X) =\mathbb{Z}$ and $X$ is a del Pezzo surface;

\item
the Picard group of $X$ is $\mathrm{Pic}(X) =\mathbb{Z} \oplus \mathbb{Z}$ and there is a morphism 
$$
\pi \colon X \to C,
$$
 where the base curve $C$ and the general fiber $X_{\eta}$ of $\pi$ are smooth curves of genus zero.

\end{enumerate}

\end{Remark}

\begin{Remark}
For the sake of convenience we call a del Pezzo surface $X$  a \textit{minimal del Pezzo surface} if it is of Picard rank $1.$

\end{Remark}

Let us recall the notion of Sarkisov links. They are special transformations between Mori fibre spaces $X\to B$ and $X'\to  B'.$ There are $4$ types of Sarkisov links

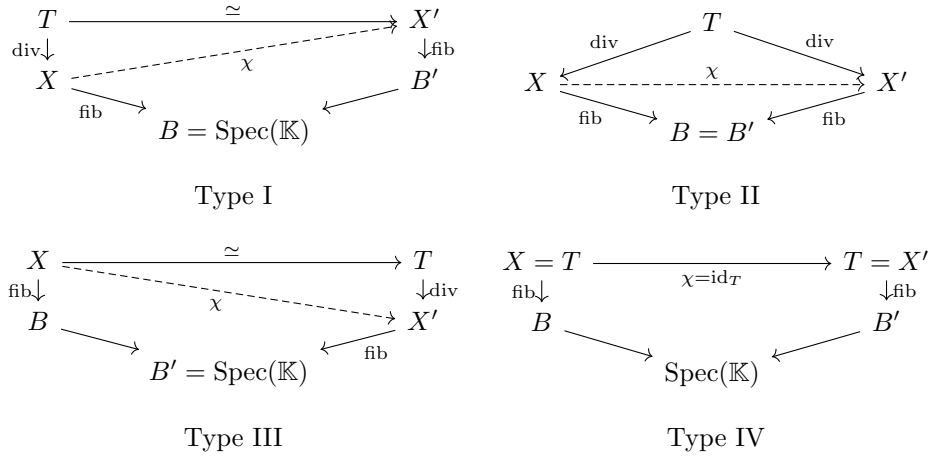
\begin{figure}[H]
\[
{
\def\arraystretch{2.2}
\begin{array}{cc}
\begin{tikzcd}[ampersand replacement=\&,column sep=1cm,row sep=0.14cm]
\ar[dd,"\rm div",swap] T \ar[rr,"\simeq"] \&\& X' \ar[dd,"\rm fib"] \\ \\
X \ar[uurr,"\chi" ,dashed,swap] \ar[dr,"\rm fib",swap] \&  \& B' \ar[dl] \\
\& B = \mathrm{Spec}(\mathbb{K}) \&
\end{tikzcd}
&
\begin{tikzcd}[ampersand replacement=\&,column sep=1.3cm,row sep=0.14cm]
  \&T\ar[ddl,"\rm div",swap]\ar[ddr,"\rm div"]\&  \\ \\
X \ar[rr,"\chi" ',dashed,swap] \ar[dr,"\rm fib",swap] \&  \& X' \ar[dl,"\rm fib"] \\
\& B = B'  \&
\end{tikzcd}
\\
\text{Type \rom{1}} & \text{Type \rom{2}} \vspace{0.2cm}
\\
\begin{tikzcd}[ampersand replacement=\&,column sep=1cm,row sep=0.14cm]
X \ar[ddrr,"\chi",dashed,swap] \ar[dd,"\rm fib",swap]  \ar[rr,"\simeq"] \&\& T \ar[dd,"\rm div"] \\ \\
B \ar[dr] \& \& X' \ar[dl,"\rm fib"] \\
\& B'= \mathrm{Spec}(\mathbb{K}) \&
\end{tikzcd}
&
\begin{tikzcd}[ampersand replacement=\&,column sep=0.8cm,row sep=0.14cm]
X=T \ar[rr,"\chi=\mathrm{id}_T",swap] \ar[dd,"\rm fib",swap]  \&\& T=X' \ar[dd,"\rm fib"] \\ \\
B \ar[dr] \& \& B'\ar[dl] \\
\& \mathrm{Spec}(\mathbb{K}) \&
\end{tikzcd}
\\
\text{Type \rom{3}} & \text{Type \rom{4}} 
\end{array}
}
\]
\caption{The four types of Sarkisov links for dimension $2.$}
\label{figure:SarkisovTypesSurfaces}
\end{figure}
\noindent where ``div'' and ``fib'' mean divisorial contraction and fibration, respectively.  Figure~\ref{figure:SarkisovTypesSurfaces} is taken from~\cite[Figure 1]{YasBlanc}.

\begin{Remark}
Sarkisov links of type~\rom{2} such that  $B=B'=\mathrm{Spec}(\mathbb{K})$ will often be encountered below.  So in this case we  draw the diagram of Sarkisov link of type~\rom{2} in the following way
\begin{equation}\label{eq:Sarkisovlink2simple}
 \begin{tikzcd}
&T \arrow[dl, "g" ']  \arrow[dr, "h"] & \\
X  \arrow[rr, "\chi", dashrightarrow] & &X', 
\end{tikzcd}
\end{equation}
\noindent where $X$ and $X'$ are two minimal del Pezzo surfaces, $g$ and $h$ are blowups of points of degrees $d$ and $d',$ respectively, in Sarkisov general position  and $T$ is a del Pezzo surface of Picard rank~$2.$ Sometimes, if $g,$ $h,$ $d,$ $d'$ and $T$ are clear from the context, we just write the Sarkisov link~\eqref{eq:Sarkisovlink2simple} as
$$
\chi \colon X \DashedArrow[->,densely dashed    ] X'.
$$

\end{Remark}

Now recall the result about Sarkisov program in class of smooth minimal rational surfaces, which is going back to V.~A.~Iskovskikh (see~\cite[Theorem~2.6]{Isk}).

\begin{Th}[{\cite[Theorem 2.4]{Lamy}}]\label{th:SarkisovlinksdP}
Let $\mathcal{D}_5,$ $\mathcal{D}_6$ and $\mathcal{D}_8$ be classes of minimal rational del Pezzo surfaces of degree $5,$ $6$ and $8,$ respectively, $\mathcal{C}_5,$ $\mathcal{C}_6$ and $\mathcal{C}_8$ be classes of minimal conic bundles such that the square of canonical class is equal to~$5,$~$6$ and~$8,$ respectively. Then we have the following diagram of Sarkisov links 

\tikzset{
  LabelStyle/.style={rectangle,sloped,rounded corners,minimum width=1em,fill=gray!40},
  VertexStyle/.append style={rectangle,inner sep=3pt,fill=red!30},
  EdgeStyle/.append style={dashed}
}

\begin{figure}[H]
\begin{center}
\begin{tikzpicture}[scale=1]
\clip(-5.4,-4.8) rectangle(7.6,5.3);
\SetGraphUnit{7.7em}
\Vertex[L=$\{\P^2\}$]{P}
\SO[L={$\Dl_5$}](P){D5}
\EA[L={$\Dl_8$}](P){D8}
\SO[L={$\Dl_6$}](D8){D6}
\tikzset{
  VertexStyle/.append style={rectangle,inner sep=3pt,fill=blue!30}
}
\NO[L={$\Cl_5$}](P){C5}
\NO[L={$\Cl_6$}](D8){C6}
\NOWE[L={$\{\F_0, \F_1, \dots\} = \Cl_8$}](P){F}

\Edge[label={I or III:1}](P)(F)
\Edge[label={I or III:4}](P)(C5)
\Edge[label={II:1:5}](D5)(P)
\Edge[label={II:2:1}](P)(D8)

\Loop[dist=8em,dir=WE,label={\rotatebox{180}{II:3:3}},style={-}](P.west)
\Loop[dist=11em,dir=WE,label={\rotatebox{180}{II:6:6}},style={-}](P.west)
\Loop[dist=14em,dir=WE,label={\rotatebox{180}{II:7:7}},style={-}](P.west)
\Loop[dist=17em,dir=WE,label={\rotatebox{180}{II:8:8}},style={-}](P.west)

\Edge[label={II:2:5}](D8)(D5)
\Edge[label={I or III:2}](D8)(C6)
\Edge[label={II:1:3}](D6)(D8)

\Loop[dist=8em,dir=EA,label={II:4:4},style={-}](D8.east)
\Loop[dist=11em,dir=EA,label={II:6:6},style={-}](D8.east)
\Loop[dist=14em,dir=EA,label={II:7:7},style={-}](D8.east)

\Loop[dist=8em,dir=WE,label={\rotatebox{180}{II:3:3}},style={-}](D5.west)
\Loop[dist=11em,dir=WE,label={\rotatebox{180}{II:4:4}},style={-}](D5.west)

\Loop[dist=8em,dir=EA,label={II:2:2},style={-}](D6.east)
\Loop[dist=11em,dir=EA,label={II:3:3},style={-}](D6.east)
\Loop[dist=14em,dir=EA,label={II:4:4},style={-}](D6.east)
\Loop[dist=17em,dir=EA,label={II:5:5},style={-}](D6.east)

\Loop[dist=7em,dir=NO,label={II:d:d},style={-}](C6.north)
\Loop[dist=7em,dir=NO,label={II:d:d},style={-}](C5.north)
\Loop[dist=7em,dir=NO,label={II:d:d},style={-}](F.north)
\Loop[dist=4em,dir=WE,label={\rotatebox{180}{IV}},style={-}](F.west)
\end{tikzpicture}
\end{center}
\caption{Sarkisov links between minimal rational surfaces over a perfect field}
\label{figure:Sarkisovlinks}
\end{figure}
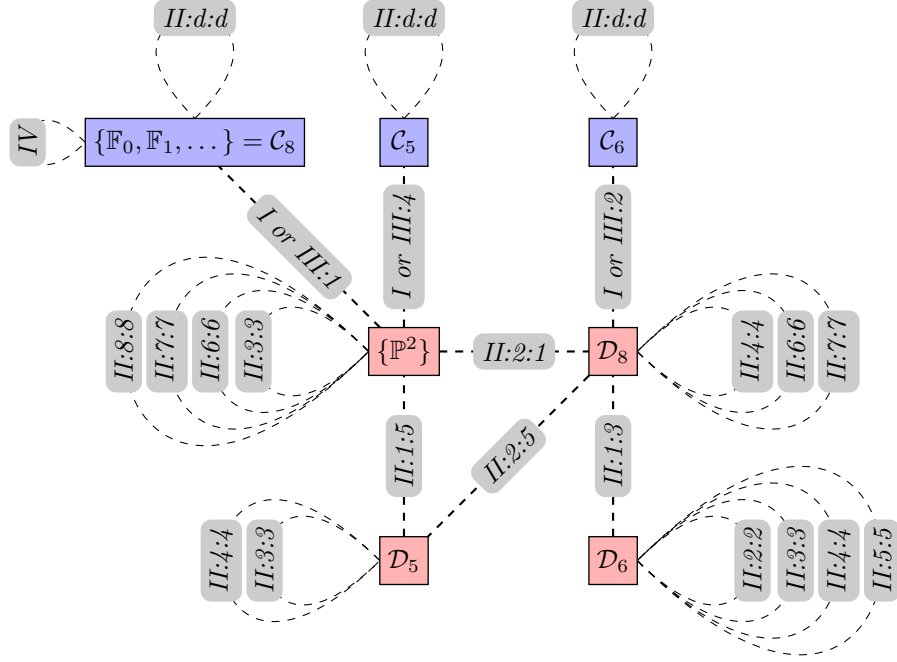
\noindent where an edge between two vertexes  means that there is a Sarkisov link between some minimal rational surfaces which classes are written in the vertexes. Roman number  in the grey box on the edge means the type of the  Sarkisov link and Arabic numbers mean the degree of points which are blown up.

\end{Th}

As we see from Theorem~\ref{th:SarkisovlinksdP} there are no Sarkisov links between minimal del Pezzo surface of degree $6$ and $\mathbb{P}^2$ and any birational map between them is a composition of at least two Sarkisov links.  Now we construct a birational map between a minimal del Pezzo surface of degree $6$  and $\mathbb{P}^2,$ which is a composition of two Sarkisov links.

\begin{Example}\label{example:dP69}
Let $\mathbb{K}$ be a perfect field, which has extensions of degree $2$ and~$3.$ For example, a field of rational numbers $\mathbb{Q}.$ We construct the birational map between $\mathbb{P}^2$ and a minimal del Pezzo surface of degree  $6$   over the field $\mathbb{K},$ which is a composition of two Sarkisov links of type~\rom{2}.

Consider projective line $\mathbb{P}^1.$ Let $Q$ be a point of degree~$2$ and $P$ be a point of degree~$3$ on $\mathbb{P}^1.$  Let 
$$
(Q)_{\overline{\mathbb{K}}}=\{Q_1, Q_2\} \quad \text{and} \quad (P)_{\overline{\mathbb{K}}}=\{P_1, P_2, P_3\}.
$$
\noindent Consider the morphism $\phi=\phi_{|-K_{\mathbb{P}^1}|} \colon \mathbb{P}^1 \to \mathbb{P}^2.$ The image $\phi(\mathbb{P}^1)$ is a conic $\mathcal{Q} \subset \mathbb{P}^2.$ Let $l$ be a line in $\mathbb{P}^2$ passing through $\phi(Q).$

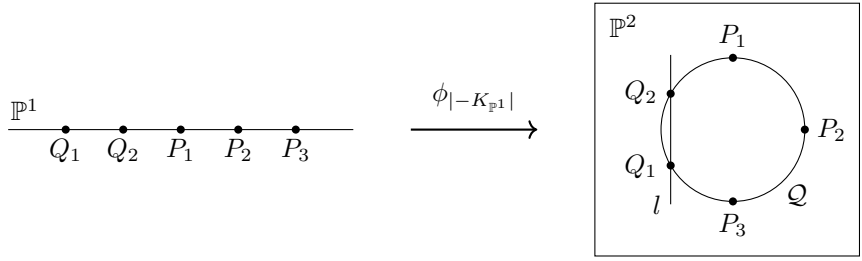
\begin{figure}[H]
\begin{center}
\begin{tikzpicture}[scale=0.76]

\draw (-6,0) -- (0,0);

\fill (-5,0) circle (2pt);
\fill (-4,0) circle (2pt);
\fill (-3,0) circle (2pt);
\fill (-2,0) circle (2pt);
\fill (-1,0) circle (2pt);

\node[above left] at (-5.3,0) {$\mathbb{P}^1$};

\node[below] at (-5,0) {$Q_1$};
\node[below] at (-4,0) {$Q_2$};
\node[below] at (-3,0) {$P_1$};
\node[below] at (-2,0) {$P_2$};
\node[below] at (-1,0) {$P_3$};

\draw[->,thick] (1.0,0) -- (3.2,0);
\node[above] at (2.1,0.15) {$\phi_{|-K_{\mathbb{P}^1}|}$};

\draw (4.2,-2.2) rectangle (8.8,2.2);

\node at (4.7,1.8) {$\mathbb{P}^2$};

\draw (6.6,0) circle (1.25);

\draw (5.517468245,-1.3) -- (5.517468245,1.3);

\node[left] at (5.52,-1.3) {$l$};

\fill (5.517468245,0.625) circle (2pt);
\fill (5.517468245,-0.625) circle (2pt);

\fill (6.6,1.25) circle (2pt);
\fill (7.85,0) circle (2pt);
\fill (6.6,-1.25) circle (2pt);

\node[left] at (5.45,0.65) {$Q_2$};
\node[left] at (5.45,-0.65) {$Q_1$};

\node[above] at (6.6,1.28) {$P_1$};
\node[right] at (7.9,0) {$P_2$};
\node[below] at (6.6,-1.3) {$P_3$};

\node[below] at (7.7,-0.8) {$\mathcal{Q}$};
\end{tikzpicture}
\end{center}
\vspace{0.5cm}
\caption{Anticanonical morphism  $\phi_{|-K_{\mathbb{P}^1}|}$}
\label{figure:dP6link}
\end{figure}

Since $5$ points $\phi(Q_1),$ $\phi(Q_2),$ $P_1,$ $P_2,$ $P_3$ are in Sarkisov general position on $\mathbb{P}^2_{\overline{\mathbb{K}}}$ we can blow them up over $\mathbb{K}$ and get a del Pezzo surface $X_4$ of degree~$4$ of Picard rank~$3.$ Let  $h_1 \colon X_4 \to \mathbb{P}^2$ be this blowup. Contracting $(h_1)_*^{-1}(l)$ and $(h_1)_*^{-1}(\mathcal{Q})$ we get a minimal del Pezzo surface $X_6$ of degree~$6.$ Let $h_2 \colon X_4 \to X_6$ be this contraction. 

 Let $g_1$ be a blowup of $\phi(Q)$ and $g_2$ be a contraction of $(-1)$-curve~$(g_1)^{-1}_*(l).$  Let $g_3$ be a blowup of   the point~$g_2\left(g_1^{-1}\left(\phi\left(P\right)\right)\right)$ and $g_4$ be a contraction of the $(-1)$-curve $(g_3)^{-1}_*\left((g_2)_*\left((g_1)^{-1}_*\left(\mathcal{Q}\right)\right)\right).$ We have the following diagram 
$$
 
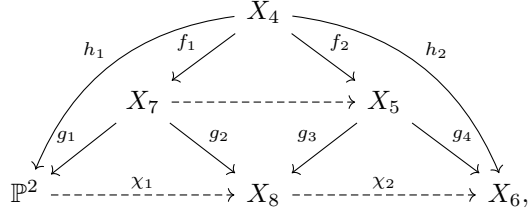
\begin{tikzcd}
&& X_4 \arrow[dl, "f_1" ']  \arrow[dr, "f_2"] \arrow[drdr, bend left, "h_2"] \arrow[dldl, bend right, "h_1"'] && \\
& X_7 \arrow[dl, "g_1" ']  \arrow[dr, "g_2"] \arrow[rr,  dashrightarrow] & & X_5 \arrow[dl, "g_3" '] \arrow[dr, "g_4"] \\
\mathbb{P}^2 \arrow[rr,  dashrightarrow, "\chi_1"] & & X_8 \arrow[rr,  dashrightarrow, "\chi_2"]  && X_6,
\end{tikzcd}
$$
where the morphism $f_1$ is a blowup of $g_1^{-1}(\phi(Q))$ and $h_1=g_1 \circ f_1;$ the morphism~$f_2$ is a contraction of $(h_1)_*^{-1}(l)$ and $h_2=g_4 \circ f_2.$ Since $X_4$ is a del Pezzo surface of degree~$4$ of Picard rank $3,$ we get that $X_7$ and $X_5$ are del Pezzo surfaces of degree $7$ and $5,$ respectively, of Picard rank~$2.$ This means that $\chi_1$ and $\chi_2$ are two Sarkisov links of type~\rom{2} and the composition of $\chi_1$ and $\chi_2$ gives us a birational map between~$\mathbb{P}^2$ and the minimal del Pezzo surface $X_6$ of degree $6.$

\end{Example}

\begin{Cor}
Let $\mathbb{K}$ be a perfect field. Then there is a minimal rational del Pezzo surface of degree~$6$ over~$\mathbb{K}$  if and only if $\mathbb{K}$ has extensions of degree~$2$ and~$3.$ 

\end{Cor}

\begin{proof}
If $\mathbb{K}$ has  extensions of degree~$2$ and~$3,$ then Example~\ref{example:dP69} gives us minimal rational del Pezzo surface of degree~$6$ over~$\mathbb{K}.$ Assume that there is a minimal rational del Pezzo surface $X_6$ of degree~$6.$ Let $\mathbf{F}$ be a minimal extension of $\mathbb{K}$ such that all $(-1)$-curves on $(X_6)_{\mathbf{F}}$ are defined. Then by~\cite[\S 4]{JZim} we get that $\mathrm{Gal}(\mathbf{F}/\mathbb{K})$ is isomorphic either to $\mathbb{Z}/6\mathbb{Z},$ or to $\mathfrak{S}_3,$ or to  $\mathfrak{D}_6.$ This means that~$\mathbb{K}$  has extensions of degree~$2$ and~$3.$ 

\end{proof}

\begin{Cor}\label{corollary:dp>2}
There is a perfect field~$\mathbb{K}$ and two birationally equivalent minimal del Pezzo surfaces $X$ and~$X'$  over~$\mathbb{K}$ such that any birational map $f \colon X \DashedArrow[->,densely dashed    ] X'$ is a composition of at least two Sarkisov links.

\end{Cor}

\begin{proof}

From Theorem~\ref{th:SarkisovlinksdP} we get that there are no Sarkisov links between minimal del Pezzo surface of degree $6$  and $\mathbb{P}^2.$ From Example~\ref{example:dP69} we obtain the existence of birationally equivalent minimal del Pezzo surface of degree $6$  and  $\mathbb{P}^2$ over some field~$\mathbb{K}.$

\end{proof}

 Let us recall well-known fact about non-trivial Severi--Brauer surfaces.

\begin{Lemma}[{\cite[Corollary~2.10]{ShBir}}]\label{lemma:SB}
Let $X$ be a  del Pezzo surface of degree $9$ without points.  Let~$X'$ be a minimal del Pezzo surface  which is birational to $X.$ Then there is a birational map~\mbox{$f : X \DashedArrow[->,densely dashed    ] X'$} such that $f$ is either an isomorphism, or a Sarkisov link of type~\rom{2}.

\end{Lemma}

\section{Del Pezzo surfaces of degree six}\label{section:dP6}

Let us recall the description of the geometry of  del Pezzo surfaces of degree $6.$  We follow the paper~\cite{Yas}.  Let us define invariants of a minimal del Pezzo surface of degree $6$ over~$\mathbb{K}.$    Recall that on $X_{\overline{\mathbb{K}}}$ there are six $(-1)$-curves. Denote them by 
$$
E_1, \, E_2, \, E_3, \, F_1, \, F_2, \, F_3.
$$
\noindent Their intersections are:
\begin{gather*}
E_i^2=F_i^2=-1 \quad \text{for} \quad  i \in \{1,2,3\};  \\
E_i \cdot E_j=F_i \cdot F_j=0 \quad \text{for} \quad  i \neq j;  \\
E_i \cdot F_j=1 \quad \text{for} \quad  i \neq j;  \\
E_i \cdot F_i=0 \quad \text{for} \quad   i \in \{1,2,3\}. 
\end{gather*}

\noindent  Let $\mathbf{F}$ be a minimal extension of $\mathbb{K}$ over which all $(-1)$-curves on $X$ are defined. The Galois group $\mathrm{Gal}(\mathbf{F}/\mathbb{K})$ acts on the set of triples 
$$
\{\{E_1, E_2, E_3\}, \{F_1, F_2, F_3\}\}
$$
with transitive action. Let~$\mathbb{K} \subset \mathbf{K} \subset \mathbf{F}$ be a minimal extension of~$\mathbb{K}$ such  that both triples 
$$
\{E_1, E_2, E_3\} \quad \text{and}  \quad \{F_1, F_2, F_3\}
$$
\noindent are defined over $\mathbf{K}.$ This means that there are two contractions
$$
 \begin{tikzcd}
&X_{\mathbf{K}} \arrow[dl, "\eta" ']  \arrow[dr, "\eta^{\text{opp}}"] & \\
S  \arrow[rr, "\chi", dashrightarrow]  & &S^{\text{opp}}, 
\end{tikzcd}
$$
such that the map $\eta$ contracts the triple $\{E_1, E_2, E_3\}$ and the map $\eta^{\text{opp}}$ contracts the triple~\mbox{$\{F_1, F_2, F_3\}.$} Note that $\chi$ is the Sarkisov link of type~\rom{2} between Severi--Brauer surfaces (cf.~\cite[Example~53]{Kollar}).  It is obvious that the extension~\mbox{$\mathbb{K} \subset \mathbf{K}$} is of degree~$2.$ Denote by $V$ one of the surfaces $S$ and $S^{\text{opp}}.$

The Galois group $\mathrm{Gal}(\mathbf{F}/\mathbb{K})$ acts on the set of pairs
$$
\{\{E_1, F_1\}, \{E_2, F_2\}, \{E_3, F_3\}\}
$$
 with transitive action.  Let $\mathbf{L}$ be a minimal extension $\mathbb{K} \subset \mathbf{L} \subset \mathbf{F},$ such that every pair~\mbox{$\{E_i,F_i\}$} for~\mbox{$i \in \{1,2,3\}$} is defined over $ \mathbf{L}.$ This means that there are three contractions
$$
\xymatrix{
&C_1\times C_2&\\
&X_{\mathbf{L}}\ar@{->}[u]_{\nu_1}\ar@{->}[ld]_{\nu_2}\ar@{->}[rd]^{\nu_3} & \\
C_2\times C_3&& C_3\times C_1
}
$$
such that $\nu_i \colon X_{\mathbf{L}} \to Y_i=C_j \times C_k$ for  $\{i,j,k\} = \{1,2,3\}$ is a contraction of~\mbox{$\{E_i,F_i\}$} to the product of conics (cf.~\mbox{\cite[Lemma~3.2]{Trepalin}}). The extension $\mathbb{K} \subset \mathbf{L}$ is of degree~$3$ or~$6$ (see~\cite[Remark~3.5]{Yas}). Denote by $Y$  one of $Y_i$ for $i \in \{1,2,3\}.$ Note that
$$
\mathbf{F}=\mathbf{K}\mathbf{L}.
$$

\begin{Remark}
In the paper~\cite{Yas} three extensions $\mathbb{K} \subset \mathbf{L}_i$ are defined such that~$\mathbf{L}_i$ is a minimal extension over which the pair $\{E_i,F_i\}$ is defined for~\mbox{$i \in \{1,2,3\}.$} While these extensions are not always Galois extensions, the extension $\mathbb{K} \subset \mathbf{L}$ always is. The field~$\mathbf{L}$ is a minimal Galois extension of $\mathbb{K}$ containing all~$\mathbf{L}_i$ for~\mbox{$i \in \{1,2,3\}.$} Also note that in the paper~\cite{Yas}  all results about Severi--Brauer data we need are proven using the field $\mathbf{L}.$

\end{Remark}

\begin{definition}\label{def:SBdata}
For a minimal del Pezzo surface $X$ of degree $6$ over~$\mathbb{K}$  we define a quintuple   
$$
(\mathbf{F}, \mathbf{K}, V, \mathbf{L}, Y),
$$
where the fields $\mathbf{F}$, $\mathbf{K}$, $\mathbf{L},$ the Severi--Brauer surface $V$ over $\mathbf{K}$ and a product of conics $Y$ over $\mathbf{L}$ are defined above. This quintuple is called \textit{Severi--Brauer data} of~$X.$
\end{definition}

\begin{definition}
Let $X$ and $X'$ be two minimal del Pezzo surfaces  with Severi--Brauer data $(\mathbf{F}, \mathbf{K}, V, \mathbf{L}, Y)$ and $(\mathbf{F}', \mathbf{K}', V', \mathbf{L}', Y'),$ respectively.  We say that these two data are \textit{equivalent} if 
$$
\mathbf{F}=\mathbf{F}', \; \mathbf{K}=\mathbf{K}', \;  \mathbf{L}=\mathbf{L}',
$$
the Severi--Brauer surface~$V$ is~$\mathbf{K}$-birational to $V'$ and $Y$ is $\mathbf{L}$-birational to $Y'.$
\end{definition}

\begin{Remark}\label{remark:equivalentSeveriBrauerdata}

As we see, the Severi--Brauer data of a minimal del Pezzo surface $X$ of degree~$6$ are not uniquely determined. However, they are uniquely determined up to equivalence. Moreover,  two minimal del Pezzo surface of degree $6$ are isomorphic if and only if their Severi--Brauer data are equivalent (see~\cite[Theorems~3.10,~3.14 and~3.19]{Yas}).

\end{Remark}

 More details on  Severi--Brauer data can be found in~\mbox{\cite[Sections 3.2 and 6]{Yas}.} A direct consequence of the definition of Severi--Brauer data is the following.

\begin{Lemma}\label{lemma:extensiondeg2dP6unique}
Let $X$ be a minimal del Pezzo surface of degree $6$ over a field $\mathbb{K}.$ Let 
$$
(\mathbf{F}, \mathbf{K}, V, \mathbf{L}, Y)
$$
 be  Severi--Brauer data for $X.$ Assume that $\mathbb{K} \subset \mathbf{K}'$ is a quadratic extension with the property that the Picard rank of $X_{\mathbf{K}'}$ is $2.$ Then  $\mathbf{K}'=\mathbf{K}.$ 
\end{Lemma}

\begin{proof}
Assume that $\mathbf{K}' \neq \mathbf{K}.$ The extension $\mathbb{K} \subset \mathbf{K}'$ is Galois. This means that there are two contractions on minimal del Pezzo surfaces of the same degree $d'$. By Theorem~\ref{th:SarkisovlinksdP} we obtain $d'=9.$ So by Remark~\ref{remark:equivalentSeveriBrauerdata} we get that $\mathbf{K}=\mathbf{K}'.$

\end{proof}

For a variety $X$ we denote by $\mathrm{Ind}(X)$ the greatest common divisor of degrees of  points on~$X.$ Recall the following well-known property of del Pezzo surfaces of degree~$6.$ 

\begin{Lemma}[{see, for example,~\cite[Remark 5.1]{Yas}}]\label{lemma:ind6}
Let $X$ be a del Pezzo surface of degree $6$ over a field $\mathbb{K}.$  Then $\mathrm{Ind}(X)$ is equal to $1,$~$2,$~$3$ or~$6.$ Moreover,~\mbox{$\mathrm{Ind}(X)=1$} if and only if $X$ has a~$\mathbb{K}$-point.
 \end{Lemma}

Now let us give the description of birational maps from minimal del Pezzo surfaces of degree $6$  without $\mathbb{K}$-points.

\begin{Prop}[{\cite[Proposition 7.4]{Yas}}]\label{prop:yas}
Let $X$ be a minimal del Pezzo surface of degree $6$  over a field $\mathbb{K}.$ Assume that $X$ does not have  $\mathbb{K}$-points. Let  $X'$ be a minimal del Pezzo surface  over $\mathbb{K}$ which is birational to $X.$ Then~$X'$ is also a del Pezzo surface of degree $6$ and either $X \simeq X',$ or  there is a birational map~\mbox{$f : X \DashedArrow[->,densely dashed    ] X'$} which is a Sarkisov link of type~\rom{2}.

\end{Prop}

\begin{proof}

 By Lemma~\ref{lemma:ind6} we have $\mathrm{Ind}(X) \neq 1.$ By~\cite[Theorem 2.6]{Isk} we get that any Sarkisov link starting from $X$ ends up with minimal del Pezzo surface of degree~$6.$   Therefore,  by~\cite[Proposition 7.4]{Yas}, which says that any birational non-rational non-isomorphic minimal del Pezzo surfaces of degree~$6$ can be connected with a Sarkisov link, we get the desired.

\end{proof}

\section{Invariants of rational del Pezzo surfaces}\label{section:invariants}

In this section we define the invariants for rational del Pezzo surfaces of degrees~$5,$~$6$ and $8.$

\begin{Lemma}\label{lemma:dp5M}
Let $X$ be a minimal del Pezzo surface of degree $5$  over a field~$\mathbb{K}.$ Then up to isomorphism $X$ is defined by the minimal extension $\mathbb{K} \subset \mathbf{M}$  over which all~$(-1)$-curves on $X_{\overline{\mathbb{K}}}$ are defined. 
\end{Lemma}

\begin{proof}
By~\cite[Theorem~3.1.3]{Skorobogatov} there is a bijection between isomorphism classes of del Pezzo surfaces of degree~$5$ over $\mathbb{K}$ and the set of continuous homomorphisms
$$
\mathrm{Gal}\left(\overline{\mathbb{K}}/\mathbb{K}\right) \to \mathfrak{S}_5
$$
up to conjugation in $\mathfrak{S}_5.$ By~\cite[Lemma 2.3]{Zaitsev}   this bijection maps a del Pezzo surface~$X$ of degree~$5$ to  the homomorphism $\mathfrak{f}_X \colon \mathrm{Gal}\left(\overline{\mathbb{K}}/\mathbb{K}\right) \to \mathfrak{S}_5$ such that 
$$
\mathrm{Ker}(\mathfrak{f}_X)=\mathrm{Gal}\left(\overline{\mathbb{K}}/\mathbf{M}\right)
$$
 with the minimal extension~\mbox{$\mathbb{K} \subset \mathbf{M}$}  over which all~$(-1)$-curves on $X_{\overline{\mathbb{K}}}$ are defined.  By~\cite[Corollaries~8.3 and 8.5]{Zaitsev} since $X$ is a minimal del Pezzo surface of degree~$5,$  the image of~$\mathfrak{f}_X$ is one of the following subgroups in the group~$\mathfrak{S}_5:$
\begin{equation}\label{eq:subgroupS5}
 \mathfrak{S}_5, \; \mathfrak{A}_5,  \; \mathfrak{D}_5, \; \mathbb{Z}/5\mathbb{Z} \rtimes \mathbb{Z}/4\mathbb{Z}, \; \mathbb{Z}/5\mathbb{Z}.
\end{equation}
It is not hard to see that any subgroup of~$\mathfrak{S}_5$ among~\eqref{eq:subgroupS5}  is unique up to conjugation. This means that a minimal del Pezzo surface $X$ of degree~$5$ up to isomorphism is uniquely defined by $\mathrm{Ker}(\mathfrak{f}_X),$
and, therefore, $X$ is uniquely defined by the field~$\mathbf{M}.$ This completes the proof.

\end{proof}

\begin{Lemma}\label{lemma:dp6KL}
Let $X$ be a minimal rational del Pezzo surface of degree $6$ over a field~$\mathbb{K}.$  Let~\mbox{$(\mathbf{F}, \mathbf{K}, V, \mathbf{L}, Y)$} be the  Severi--Brauer data for $X.$  Then up to isomorphism~$X$ is defined by the fields $\mathbf{K}$ and $\mathbf{L}.$
\end{Lemma}

\begin{proof}
By rationality of $X$ and Theorem~\ref{theorem:Lang-Nishimura} we get that $V \simeq \mathbb{P}^2_{\mathbf{K}}$ and $Y \simeq \mathbb{P}^1_{\mathbf{L}} \times \mathbb{P}^1_{\mathbf{L}},$ as $\mathbb{K} \subset \mathbf{L}$ is a Galois extension.  Since $\mathbf{F}=\mathbf{K}\mathbf{L},$ we get the desired.

\end{proof}

Now we give a well-known classification of minimal rational del Pezzo surface of degree~$8.$

\begin{Lemma}[{\cite[Lemma 7.3(\romsmall{1})]{ShramovVologodsky}}]\label{lemma:dP8Weil}
Let $X$ be a minimal rational del  Pezzo surface of degree $8$  over $\mathbb{K}.$ Then there is a quadratic extension $\mathbb{K} \subset \mathbf{K},$ such that $X$  is isomorphic to the Weil restriction of scalars $\mathrm{R}_{\mathbf{K}/\mathbb{K}}(\mathbb{P}^1).$ Moreover, the quadratic extension~\mbox{$\mathbb{K} \subset \mathbf{K}$} is uniquely defined by $X.$
\end{Lemma}

\section{Points in Sarkisov general position}\label{section:existencepointssarkisov}

In this section we study existence of points in Sarkisov general position on minimal del Pezzo surfaces of degree~$5$ and~$6.$ First of all, we show that any $\mathbb{K}$-point on a minimal rational del Pezzo surface is in Sarkisov general position.

\begin{Lemma}[{\cite[Lemma 2.6]{BernasconiTanaka}}]\label{lemma:point>4generalpos}
Let $X$ be a minimal del Pezzo surface of degree~$d \geqslant 4$ over $\mathbb{K}.$ Then any $\mathbb{K}$-point $P \in X$ is in Sarkisov general position. 
\end{Lemma}

\begin{Cor}\label{cor:pointisingeneralpos}
Any $\mathbb{K}$-point on a minimal rational del Pezzo surface is in Sarkisov general position. 
\end{Cor}

\begin{proof}
By Theorem~\ref{th:SarkisovlinksdP} the degree of any minimal del Pezzo surface at least~$5.$  By Theorem~\ref{theorem:Lang-Nishimura} any rational del Pezzo surface has a $\mathbb{K}$-point. By Lemma~\ref{lemma:point>4generalpos} this point is in Sarkisov general position. 

\end{proof}

Now we study points in Sarkisov general position on minimal del Pezzo surfaces of degree~$5.$

\begin{Lemma}\label{lemma:dP5minimal quadraticextension}
Let $X$ be a minimal del Pezzo surface of degree~$5$ over a field~$\mathbb{K}.$ Then for any extension~$\mathbb{K} \subset \mathbf{K}$ of degree~$2$ the del Pezzo surface $X_{\mathbf{K}}$ is also minimal.
\end{Lemma}

\begin{proof}
Consider a  minimal extension $\mathbb{K} \subset \mathbf{M}$  over which all~$(-1)$-curves on $X_{\overline{\mathbb{K}}}$ are defined.
By minimality of $X$ from~\cite[Lemma 8.4]{Zaitsev} we get that $\mathrm{Gal}(\mathbf{M}/\mathbb{K})$ contains a subgroup isomorphic to~$\mathbb{Z}/5\mathbb{Z}.$ So the Galois group $\mathrm{Gal}(\mathbf{M}\mathbf{K}/\mathbf{K})$ also contains a subgroup isomorphic to~$\mathbb{Z}/5\mathbb{Z}.$ Therefore, again by~\cite[Lemma 8.4]{Zaitsev} we have that~$X_{\mathbf{K}}$ is minimal.

\end{proof}

\begin{Lemma}\label{lemma:existencepointd5p2}
Let $X$ be a minimal del Pezzo surface of degree~$5$ over~$\mathbb{K}.$ Then any point of degree~$2$ on~$X$ is in Sarkisov general position. In particular, for any quadratic extension $\mathbb{K} \subset \mathbf{K}$ there is a $\mathbf{K}$-point in Sarkisov general position on $X_{\mathbf{K}}.$
\end{Lemma}

\begin{proof}
Let $P \in X$ be a point of degree~$2$ with the splitting field~$\mathbf{K}.$ Let 
$$
(P)_{\mathbf{K}}=\{P_1, P_2\}.
$$
\noindent  By Lemma~\ref{lemma:dP5minimal quadraticextension} we get that $X_{\mathbf{K}}$ is a minimal del Pezzo surface. So by Corollary~\ref{cor:pointisingeneralpos} the point~$P_1$ is in Sarkisov general position on $X_{\mathbf{K}}.$ This means that $P_1$ does not lie on $(-1)$-curves on~$X_{\overline{\mathbb{K}}}.$ Therefore, neither does $P_2.$  

Let $\pi \colon Y \to X_{\mathbf{K}}$ be a blowup of the point~$P_1.$ Let $E$ be the exceptional divisor of~$\pi.$ It is enough to prove that the point $P'=\pi^{-1}(P_2)$ does not lie on $(-1)$-curves on~$Y_{\overline{\mathbb{K}}}.$ Since $P_2$ does not lie on $(-1)$-curves on $X_{\overline{\mathbb{K}}},$ it is enough to prove that $P'$ does not lie on the $(-1)$-curves which are not a proper transform of $(-1)$-curves under the morphism~$\pi.$ Assume that $P'$ lies on such $(-1)$-curve $E'.$ As $E \cdot E'=1,$ we get that~$\pi_*(E')$ is an irreducible curve with self-intersection $0$ defined over $\mathbb{K}.$ But this is impossible, since  $\pi_*(E')$ is not proportional to $K_X.$

\end{proof}

Now we study points in Sarkisov general position on minimal del Pezzo surfaces of degree~$6.$

\begin{Lemma}\label{lemma:dP6pointsinintersection}
Let $P$ be a point of degree~$2$ on a  minimal  del Pezzo surface $X$ of degree~$6$ over~$\mathbb{K}.$ Then $P$ does not lie on intersections of $(-1)$-curves on $X_{\overline{\mathbb{K}}}.$

\end{Lemma}

\begin{proof}
Assume that $P$  lies on intersections of $(-1)$-curves on $X_{\overline{\mathbb{K}}}.$ Let $(\mathbf{F}, \mathbf{K}, V, \mathbf{L}, Y)$ be Severi--Brauer data of $X$ (see Definition~\ref{def:SBdata}).  Consider a point~\mbox{$P \in X$}  of degree~$2$ with the splitting field~$\widetilde{\mathbf{K}}.$ Let  
$$
(P)_{\widetilde{\mathbf{K}}}=\{P_1, P_2\}.
$$
\noindent Assume that~$\mathbf{K} \neq \widetilde{\mathbf{K}}.$ Then  $X_{\widetilde{\mathbf{K}}}$ is also a minimal del Pezzo surface of degree~$6.$ By Corollary~\ref{cor:pointisingeneralpos} we get that every point among $P_1$ and $P_2$ is in Sarkisov general position and, therefore,  $P$ does not lie on intersections of $(-1)$-curves on~$X_{\overline{\mathbb{K}}}.$

Now assume that $\mathbf{K}=\widetilde{\mathbf{K}}.$ Then we get $V=\mathbb{P}^2.$  Consider the contraction 
$$
\eta \colon X_{\mathbf{K}} \to \mathbb{P}^2.
$$
\noindent Let $Q_1,$ $Q_2$ and $Q_3$ be three points on $\mathbb{P}^2_{\mathbf{F}}$ which are blown up by~$\eta_{\mathbf{F}}.$ By definition of the field~$\mathbf{K}$ we get that $\mathrm{Gal}(\mathbf{F}/\mathbf{K})$ acts transitively on the set of points~$\{Q_1, Q_2, Q_3\}.$ However, since $P$  lies on intersections of $(-1)$-curves, we get 
$$
\eta(P_1), \eta(P_2) \in \{Q_1, Q_2, Q_3\}.
$$
\noindent  This is impossible, since $\mathrm{Gal}(\mathbf{F}/\mathbf{K})$ acts trivially on the set~$\{P_1, P_2\}.$ This contradiction concludes the proof.

\end{proof}

\begin{Lemma}\label{lemma:existence2pointd6}
Let $X$ be a minimal  del Pezzo surface of degree~$6$ over~$\mathbb{K}.$ Then any point of degree~$2$ on~$X$ is in Sarkisov general position. In particular, for any quadratic extension $\mathbb{K} \subset \mathbf{K}$ there is a $\mathbf{K}$-point in Sarkisov general position on $X_{\mathbf{K}}.$
\end{Lemma}

\begin{proof}
Let $(\mathbf{F}, \mathbf{K}, V, \mathbf{L}, Y)$ be Severi--Brauer data of $X.$  Consider a point~\mbox{$P \in X$}  of degree~$2$ with the splitting field~$\widetilde{\mathbf{K}}.$ Let  
$$
(P)_{\widetilde{\mathbf{K}}}=\{P_1, P_2\}.
$$

\noindent Consider the field $\mathbf{N}=\mathbf{K}\widetilde{\mathbf{K}}.$ By definition of Severi--Brauer data any $(-1)$-curve on $X_{\overline{\mathbb{K}}}$ is not defined over~$\mathbf{N}.$ This means that either $P_1$ lies on the intersection of $(-1)$-curves on $X_{\overline{\mathbb{K}}},$ or out of  $(-1)$-curves on $X_{\overline{\mathbb{K}}}.$ If $P$ lies on the intersection of $(-1)$-curves on $X_{\overline{\mathbb{K}}},$ then the sum $\Sigma$ of $(-1)$-curves  which intersect with $P$ is defined over $\mathbb{K}.$ This means that $\Sigma$ is proportional to $-K_X,$ which is impossible. So we get that neither $P_1,$ nor $P_2$ lies on $(-1)$-curves on $X_{\overline{\mathbb{K}}}.$

Let $\pi \colon Y \to X_{\mathbf{N}}$ be a blowup of the point~$P_1.$ Let $E$ be the exceptional divisor of~$\pi.$ It is enough to prove that the point $P'=\pi^{-1}(P_2)$ does not lie on $(-1)$-curves on~$Y_{\overline{\mathbb{K}}}.$ Since $P_2$ does not lie on $(-1)$-curves on $X_{\overline{\mathbb{K}}},$ it is enough to prove that $P'$ does not lie on the $(-1)$-curves which are not a proper transform of $(-1)$-curves under the morphism~$\pi.$ Assume that~$P'$ lies on such $(-1)$-curve $E'.$ As $E \cdot E'=1,$ we get that~$\pi_*(E')$ is an irreducible curve with self-intersection $0$ defined over $\mathbb{K}.$ So the contraction of $\pi_*(E')$ gives a conic bundle structure on $X_{\mathbf{N}}.$ But this is impossible since conic bundle structures on $X_{\overline{\mathbb{K}}}$ are defined over fields which contain the field~$\mathbf{L}.$

\end{proof}

\begin{Lemma}\label{lemma:existence3pointd6}
Let $X$ be a minimal del Pezzo surface of degree~$6$ over~$\mathbb{K}.$ Let 
$$
(\mathbf{F}, \mathbf{K}, V, \mathbf{L}, Y)
$$
 be Severi--Brauer data of $X.$  Consider a point~\mbox{$P \in X$}  of degree~$3$ with the splitting field~$\widetilde{\mathbf{L}}.$ Assume that $\mathbf{L} \neq \widetilde{\mathbf{L}}.$ Then $P$ is in Sarkisov general position.

\end{Lemma}

\begin{proof}
 Let  $(P)_{\widetilde{\mathbf{L}}}=\{P_1, P_2, P_3\}.$ Consider $\eta \colon X_{\mathbf{K}} \to V.$ Denote by $Q$ the point which is blown up by~$\eta.$ Let
$$
(Q)_{\overline{\mathbb{K}}}=\{Q_1, Q_2, Q_3\}.
$$
\noindent Denote by $E_1,$ $E_2,$ $E_3$ three exceptional curves of $\eta_{\overline{\mathbb{K}}}$ and by $F_1,$ $F_2,$ $F_3$ the other $(-1)$-curves on $X_{\overline{\mathbb{K}}}.$ Let us prove that there are no lines passing through some three points among 
\begin{equation}\label{eq:3+3points}
\eta(P_1), \eta(P_2), \eta(P_3), Q_1, Q_2, Q_3
\end{equation}
\noindent on $V_{\overline{\mathbb{K}}}.$ Arguing by contradiction, suppose that there is such line. By the construction, we see that the points  $Q_1,$ $Q_2$ and $Q_3$ do not lie on a line. Assume that 
$$
\eta(P_i), \, \eta(P_j) \; \text{and} \; Q_k \;\;  \text{for} \; i,j,k \in \{1,2,3\} \; \text{and} \;  i \neq j
$$
\noindent lie on a line $L.$ Then $L$ is defined over $\widetilde{\mathbf{L}}$ on $V_{\widetilde{\mathbf{L}}} \simeq \mathbb{P}^2$ by definition of $P.$ However, since the point $Q_k$ is not defined over~$\widetilde{\mathbf{L}}$ we get that $Q$ lies on $L,$ which is impossible.

Now assume that $$
\eta(P_i), \, Q_j \; \text{and} \; Q_k \;\;  \text{for} \; i,j,k \in \{1,2,3\} \; \text{and} \;  j \neq k
$$
\noindent lie on a line $L.$ Consider  $\mathrm{Gal}(\widetilde{\mathbf{L}}\mathbf{L}/\widetilde{\mathbf{L}})$-action of $L.$ Let $\tau$ be a non-trivial element in the Galois group $\mathrm{Gal}(\widetilde{\mathbf{L}}\mathbf{L}/\widetilde{\mathbf{L}}).$ Then on the one hand we have $L \cdot \tau{L}=1,$ but on the other hand $L \cdot \tau{L} \supset \{\eta(P_i), Q_l\},$ where $Q_l$ is either  $Q_j,$ or $Q_k.$ This means that $Q$ lies on $L,$ which is impossible.

Now assume that $\eta(P)$ lies on a line~$L.$  Consider $\eta_*^{-1}(L) \subset X_{\mathbf{K}}.$ By definition of~$\eta$ we get 
$$
\eta_*^{-1}(L) \cdot F_i =1; \; \eta_*^{-1}(L) \cdot E_i =0 \; \text{for} \; i=1,2,3 \;\; \text{and} \;\; (\eta_*^{-1}(L))^2=1.
$$
\noindent   Consider a non-trivial element~\mbox{$\iota \in \mathrm{Gal}(\mathbf{K}/\mathbb{K}).$} So we obtain 
$$
\iota(\eta_*^{-1}(L)) \cdot E_i=1 \; \text{for} \; i=1,2,3.
$$
\noindent This means that $\eta_*(\iota(\eta_*^{-1}(L)))$ passes through $Q,$ and, as a result, 
$$
\eta_*(\iota(\eta_*^{-1}(L)))^2=4,
$$
which means that $\eta_*(\iota(\eta_*^{-1}(L)))$ is a conic passing through $6$ points~\eqref{eq:3+3points}. Moreover, we get
$$
P \in L \cap \eta_*(\iota(\eta_*^{-1}(L))) 
$$
and, thus, $L \cdot \eta_*(\iota(\eta_*^{-1}(L))) \geqslant 3,$ which contradicts with the fact that  $\eta_*(\iota(\eta_*^{-1}(L)))$ is a conic.

 Similarly, we get the contradiction with the assumption that all points in~\eqref{eq:3+3points} lies on a conic $\mathcal{Q}.$ Indeed, using similar arguments as before we get that 
$$
\eta_*(\iota(\eta_*^{-1}(\mathcal{Q})))=L
$$
 is a line and $P \in L \cap \mathcal{Q}$ which is impossible. As a result we get that $P$ is in Sarkisov general position.

\end{proof}

\section{Sarkisov links}\label{section:existencesarkisov}

In this section we study existence of Sarkisov links of type~\rom{2} which appears in the next section as well as the behaviour  of fields of definition of $(-1)$-curves under Sarkisov links of type~\rom{2}.

\begin{Lemma}\label{lemma:linkd9d5}

 Let $X$ be a minimal del Pezzo surface of degree~$5$ over~$\mathbb{K},$ which is defined by a field $\mathbf{M}.$  Then there is a Sarkisov link of type~\rom{2}
$$
 \begin{tikzcd}
&T \arrow[dl, "g" ']  \arrow[dr, "h"] & \\
\mathbb{P}^2   \arrow[rr, "\chi", dashrightarrow] & &X  
\end{tikzcd}
$$
\noindent between $\mathbb{P}^2$ and $X$ such that $g$ is a blowup of a point of degree $5$ with the splitting field $\mathbf{M}$ and  $h$ is a blowup of a $\mathbb{K}$-point.

\end{Lemma}

\begin{proof}
The existence of a~$\mathbb{K}$-point  on~$X$ follows from Theorem~\ref{th:dp5point}. By Lemma~\ref{lemma:point>4generalpos} any $\mathbb{K}$-point is in Sarkisov general position.  So the existence of the Sarkisov link follows from Theorem~\ref{th:SarkisovlinksdP}.

Let $M'$ be a splitting field of $P'.$ Then on the one hand, we get that the  Picard rank of $T_{\mathbf{M}'}$ is equal to $6$ and, therefore,  the Picard rank of $X_{\mathbf{M}'}$ is equal to~$5.$ So by minimality of $\mathbf{M}$ we get that $\mathbf{M} \subseteq \mathbf{M}'.$ But on the other hand, since the Picard rank of $X_{\mathbf{M}}$ is equal to~$5,$  the  Picard rank of $T_{\mathbf{M}}$ is also equal to $6.$ So we get that the point $P'$ also splits over $\mathbf{M},$ so  $\mathbf{M}' \subseteq \mathbf{M}.$ Thus, $\mathbf{M}' = \mathbf{M}.$

\end{proof}

\begin{Lemma}\label{lemma:linkd8d5}
 Let $X$ be a minimal del Pezzo surface of degree~$5$ over~$\mathbb{K}$ and $X'$ be a minimal del Pezzo surface of degree~$8$ over~$\mathbb{K}.$ Then there is a Sarkisov link of type~\rom{2}
$$
 \begin{tikzcd}
&T \arrow[dl, "g" ']  \arrow[dr, "h"] & \\
X \arrow[rr, "\chi", dashrightarrow] & & X'
\end{tikzcd}
$$
\noindent between $X$ and $X'.$ Assume that $X$ is defined by a field $\mathbf{M}$ and $X'$ is defined by a field $\mathbf{K}'.$ Then $g$ is a blowup of a point of degree $2$ with splitting field~$\mathbf{K}'$ and $h$  is a blowup of a point of degree $5$ with splitting field $\mathbf{M}.$

\end{Lemma}

\begin{proof}

By Lemma~\ref{lemma:existencepointd5p2} any point of degree $2$  for any quadratic extension of $\mathbb{K}$ on~$X$ is in Sarkisov general position. So the existence of the Sarkisov link follows from Theorem~\ref{th:SarkisovlinksdP}. The construction of the Sarkisov link can be found in~\cite[Theorem~2.6]{Isk}.

Let $P'$ be a point of degree $5$ with a splitting field~$\mathbf{M}'$ on the del Pezzo surface $X'$ of degree $8,$ which is blown up by $h.$ Let $\mathcal{E}_1$ and $\mathcal{E}_2$ be irreducible $(2,1)$-divisor and irreducible $(1,2)$-divisor, respectively, on~\mbox{$X'_{\overline{\mathbb{K}}}$} passing through $P'.$ Since 
$$
\dim H^0\left(X'_{\overline{\mathbb{K}}},\mathcal{E}_1\right)=\dim H^0\left(X'_{\overline{\mathbb{K}}},\mathcal{E}_2\right)=6
$$
and since $P'$ is in general position, we get that $\mathcal{E}_1$ and $\mathcal{E}_2$ are uniquely defined. Also note that since $X' \simeq \mathrm{R}_{\mathbf{K}'/\mathbb{K}}(\mathbb{P}^1),$ we have that $\mathcal{E}_1$ and $\mathcal{E}_2$ are $\mathrm{Gal}\left(\mathbf{K}'/\mathbb{K}\right)$-conjugate to each other. By definition of $\mathcal{E}_1$ and $\mathcal{E}_2$ we get that $h_*^{-1}\left(\mathcal{E}_1\right)$ and~$h_*^{-1}\left(\mathcal{E}_2\right)$ are two skew $(-1)$-curves on $T_{\overline{\mathbb{K}}}.$  Since $h_*^{-1}\left(\mathcal{E}_1+\mathcal{E}_2\right)$ is defined over $\mathbb{K}$ we contract it with~$g$ and get a minimal del Pezzo surface~$X$ of degree $5,$  as Picard rank of $T$ is equal to~$2.$ So the splitting field of a point on~$X$ which is blown up by $g$ is equal to~$\mathbf{K}'.$ 

Now let us consider all $(1,1)$-divisors passing through any three points among
$$
P'_{\overline{\mathbb{K}}}=\{P'_1, P'_2, P'_3, P'_4, P'_5\}
$$
on $X'_{\overline{\mathbb{K}}}.$ Since the dimension of global sections of $(1,1)$-divisor on $X'_{\overline{\mathbb{K}}}$ is equal to~$4$ we get that for any $3$ points in general position there is a unique $(1,1)$-divisor passing through these points. Denote by $\mathcal{L}_{ijk}$ the irreducible $(1,1)$-divisor on  $X'_{\overline{\mathbb{K}}}$ passing through $P'_i,$ $P'_j$ and $P'_k$ for pairwise different  $i,$ $j$ and $k.$ Such $\mathcal{L}_{ijk}$  for any pairwise different  $i,j,k \in \{1,2,3,4,5\}$ exist since $P'$ is in Sarkisov general position. By the uniqueness of $\mathcal{L}_{ijk}$ we get that any $\mathcal{L}_{ijk}$ is defined over splitting field of~$P'.$ Note that $g_*\left(h_*^{-1}(\mathcal{L}_{ijk})\right)$ is a $(-1)$-curve on $X_{\overline{\mathbb{K}}}$ for any pairwise different  $i,$ $j$ and~$k.$ Since there are exactly ten $(-1)$-curves on $X_{\overline{\mathbb{K}}}$ and $C^3_5=10,$ we get that any $(-1)$-curve on  $X_{\overline{\mathbb{K}}}$ is of the form 
$$
g_*\left(h_*^{-1}(\mathcal{L}_{ijk})\right) \;\; \text{for pairwise different} \;\; i,j,k \in \{1,2,3,4,5\}.
$$ 
\noindent By definition of the field~$\mathbf{M}$  it is the minimal field over which all $g_*\left(h_*^{-1}(\mathcal{L}_{ijk})\right)$ are defined. On the other hand, all $g_*\left(h_*^{-1}(\mathcal{L}_{ijk})\right)$ are also defined over the splitting field of $P',$ since all $\mathcal{L}_{ijk}$ are  defined over the splitting field of $P'.$ This means that~\mbox{$\mathbf{M} \subseteq \mathbf{M}'.$} On the other hand, since any $\mathcal{L}_{ijk}$ is defined over $\mathbf{M},$ this means that any triple $\{P'_i, P'_j, P'_k\}$ is defined over $\mathbf{M}$ for any  pairwise different  $i,$ $j$ and~$k$ in $\{1,2,3,4,5\}.$ This means that any $P'_l$ for $l \in \{1,2,3,4,5\}$ is defined over $\mathbf{M},$ thus~\mbox{$\mathbf{M}' \subseteq \mathbf{M}.$} Therefore, we get the desired equality~$\mathbf{M}=\mathbf{M}'.$

\end{proof}

The next three lemmas give us the existence of Sarkisov links of type~\rom{2} starting from a minimal rational del Pezzo surface of degree~$6.$

\begin{Lemma}\label{lemma:linkd8d6}
 Let $X$  be a minimal rational del Pezzo surface of degree~$6$ over~$\mathbb{K}.$ Then there is a Sarkisov link of type~\rom{2}
$$
 \begin{tikzcd}
&T \arrow[dl, "g" ']  \arrow[dr, "h"] & \\
X   \arrow[rr, "\chi", dashrightarrow] & &X ' 
\end{tikzcd}
$$
\noindent between $X$ and a minimal del Pezzo surface $X'$ of degree~$8.$   Assume that $X$ is defined by a pair of fields $\left(\mathbf{K}, \mathbf{L} \right)$ and~$X'$ is defined by a field $\mathbf{K}'.$ Then $\mathbf{K}=\mathbf{K}',$  the morphism $g$  is a blowup of a $\mathbb{K}$-point  and  $h$ is a blowup of a point of degree $3$ with a splitting field $\mathbf{L}.$

\end{Lemma}

\begin{proof}
By Corollary~\ref{cor:pointisingeneralpos} we get that any $\mathbb{K}$-point $X$ is in Sarkisov general position. So the existence of the Sarkisov link follows from Theorem~\ref{th:SarkisovlinksdP}. The construction of the Sarkisov link can be found in~\cite[Theorem~2.6]{Isk}.

Now let us prove that $\mathbf{K}=\mathbf{K}'.$ By definition of the field~$\mathbf{K}'$ (i.e. the field which is defined in Lemma~\ref{lemma:dP8Weil}) we get $X'_{\mathbf{K}'} \simeq \mathbb{P}^1_{\mathbf{K}'} \times \mathbb{P}^1_{\mathbf{K}'}.$ So this means that the Picard rank of the surface $T_{\mathbf{K}'}$ is more than or equal to $3.$ Since by Theorem~\ref{th:SarkisovlinksdP} the map $g$ is a blowup of a  $\mathbb{K}$-point, we get that the Picard rank of~$X_{\mathbf{K}'}$ is $2.$ So by Lemma~\ref{lemma:extensiondeg2dP6unique} we obtain~$\mathbf{K}=\mathbf{K}'.$

Finally, we need to prove that $h$ is a blowup of a point of degree $3$ with a splitting field $\mathbf{L}.$ Denote by $P'$ a point of degree~$3$ on the del Pezzo surface  $X'$ of degree $8$ such that $h$ is a blowup in $P'.$  Let $\widetilde{\mathbf{L}}$ be a splitting field of the point $P'.$ Let 
$$
P'_{\widetilde{\mathbf{L}}}=\{P'_1,P'_2,P'_3\}.
$$
 Denote by $\mathcal{E}_i$ and $\mathcal{F}_i$ the~$(1,0)$-divisor and $(0,1)$-divisor, respectively, on $X'_{\overline{\mathbb{K}}}$ passing through $P'_i$ for~\mbox{$i \in \{1,2,3\}$} and denote by $\mathcal{L}$ the $(1,1)$-divisor on $X$ passing through~$P'.$ Since $g$ is a contraction of $h_*^{-1}(\mathcal{L}),$  we obtain $(-1)$-curves~\mbox{$g_*\left(h_*^{-1}(\mathcal{E}_i)\right)$} and $g_*\left(h_*^{-1}(\mathcal{F}_i)\right)$ on $X_{\overline{\mathbb{K}}}.$  Note that their intersections are:
\begin{gather*}
g_*\left(h_*^{-1}(\mathcal{E}_i)\right) \cdot g_*\left(h_*^{-1}(\mathcal{E}_j)\right)=g_*\left(h_*^{-1}(\mathcal{F}_i)\right) \cdot g_*\left(h_*^{-1}(\mathcal{F}_j)\right) =0 \quad \text{for} \quad  i \neq j;  \\
g_*\left(h_*^{-1}(\mathcal{E}_i)\right) \cdot g_*\left(h_*^{-1}(\mathcal{F}_j)\right)=1 \quad \text{for} \quad  i \neq j;  \\
g_*\left(h_*^{-1}(\mathcal{E}_i)\right)  \cdot g_*\left(h_*^{-1}(\mathcal{F}_i)\right)=0 \quad \text{for} \quad   i \in \{1,2,3\}. 
\end{gather*}
So by definition of Severi--Brauer data (see Definition~\ref{def:SBdata}) we get that the minimal field over which the pairs 
$$
\{g_*\left(h_*^{-1}(\mathcal{E}_i)\right),  g_*\left(h_*^{-1}(\mathcal{F}_i)\right)\}
$$
for~\mbox{$i \in \{1,2,3\}$} are defined is $\mathbf{L}.$ Since  the pairs $\{\mathcal{E}_i, \mathcal{F}_i\}$ for~\mbox{$i \in \{1,2,3\}$} are defined over~$\widetilde{\mathbf{L}},$ we get $\widetilde{\mathbf{L}}=\mathbf{L}.$

\end{proof}

\begin{Lemma}\label{lemma:linkd6d6p2}
 Let $X$  be a minimal rational del Pezzo surface of degree~$6$ over~$\mathbb{K}.$ Then there is a  Sarkisov link of type~\rom{2}
$$
 \begin{tikzcd}
&T \arrow[dl, "g" ']  \arrow[dr, "h"] & \\
X  \arrow[rr, "\chi", dashrightarrow] & & X'
\end{tikzcd}
$$
\noindent between $X$ and a minimal del Pezzo surface $X'$ of degree~$6,$ such that $g$ and $h$ are blowups of a point of degree~$2.$ Assume that $X$ is defined by a pair of fields $\left(\mathbf{K}, \mathbf{L} \right)$ and~$X'$ is defined by a pair of fields $\left(\mathbf{K}', \mathbf{L}' \right).$ If $g$ is a blowup of a point~$P$ of degree~$2$ with splitting field $\mathbf{K}'',$ then $\mathbf{K}'=\mathbf{K}''$ and $\mathbf{L}=\mathbf{L}'.$

\end{Lemma}

\begin{proof}

By Lemma~\ref{lemma:existence2pointd6} any point of degree $2$ on a minimal del Pezzo surface of degree~$6$ is in Sarkisov general position.  So the existence of the Sarkisov link follows from Theorem~\ref{th:SarkisovlinksdP}. The construction of the Sarkisov link can be found in~\cite[Theorem~2.6]{Isk}. The equalities of the fields follows from~\cite[Corollary~6.7]{Yas}.

\end{proof}

\begin{Lemma}\label{lemma:linkd6d6p3}
 Let $X$  be a minimal rational del Pezzo surface of degree~$6$ over~$\mathbb{K}.$  Then there is a  Sarkisov link of type~\rom{2}
$$
 \begin{tikzcd}
&T \arrow[dl, "g" ']  \arrow[dr, "h"] & \\
X  \arrow[rr, "\chi", dashrightarrow] & & X'
\end{tikzcd}
$$
\noindent between $X$ and a minimal del Pezzo surface $X'$ of degree~$6,$ such that $g$ and $h$ are blowups of a point of degree~$3.$ Assume that $X$ is defined by a pair of fields $\left(\mathbf{K}, \mathbf{L} \right)$ and~$X'$ is defined by a pair of fields $\left(\mathbf{K}', \mathbf{L}' \right).$ Assume that $\mathbf{L} \neq \mathbf{L}'.$ If $g$ is a blowup of a point $P$ of degree~$3$ with splitting field $\mathbf{L}'',$ then $\mathbf{L}'=\mathbf{L}''$ and~\mbox{$\mathbf{K}=\mathbf{K}'.$}

\end{Lemma}

\begin{proof}
By Lemma~\ref{lemma:existence3pointd6} any point of degree $3$ of splitting field $\mathbf{L}'' \neq \mathbf{L}$ on $X$ is in Sarkisov general position.  So the existence of the Sarkisov link follows from Theorem~\ref{th:SarkisovlinksdP}. The construction of the Sarkisov link can be found in~\cite[Theorem~2.6]{Isk}. The equalities of the fields follows from~\cite[Corollary 6.11]{Yas}.

\end{proof}

The last lemma of this section is  well-known fact about the stereographic projection.

\begin{Lemma}\label{lemma:linkd9d8}
Let $\mathbb{K} \subset \mathbf{K}$ be a quadratic field extension. Let $X \simeq \mathrm{R}_{\mathbf{K}/\mathbb{K}}(\mathbb{P}^1).$ Then there is a Sarkisov link of type~\rom{2}
$$
 \begin{tikzcd}
&T \arrow[dl, "g" ']  \arrow[dr, "h"] & \\
\mathbb{P}^2   \arrow[rr, "\chi", dashrightarrow] & &X  
\end{tikzcd}
$$
\noindent between $\mathbb{P}^2$ and $X$ such that $g$ is a blowup of a point of degree $2$ with a splitting field $\mathbf{K}.$

\end{Lemma}

\begin{proof}
Let $h \colon T \to X$ be a blowup of a $\mathbb{K}$-point $P.$ Let $E_1$ and $E_2$ be two lines on~$X_{\mathbf{K}}$ passing through $P.$ Then every line among  $E_1$ and $E_2$ is defined over $\mathbf{K}.$ Also note that  $E_1+E_2$ is defined over $\mathbb{K}$ and $h^{-1}_*(E_1+E_2)$ is two skew $(-1)$-curves on $T_{\mathbf{K}}.$ Let $g$ be a contraction of $h^{-1}_*(E_1+E_2).$ So we get $g \colon T \to \mathbb{P}^2.$ So as~\mbox{$h^{-1}_*(E_1+E_2)$} is two skew $(-1)$-curves on $T_{\mathbf{K}}$ we get that $g$ is a blowup of a point of degree $2$ with the splitting field $\mathbf{K}.$

\end{proof}

\section{Minimal number of Sarkisov links}\label{section:sarkisov}

In this section for any minimal rational del Pezzo surfaces $X$ and $X'$ we find a minimal number of Sarkisov links such that their composition gives  the birational map between $X$ and $X'.$

\begin{Lemma}\label{lemma:d>5n<2}
 Let $X$ and $X'$ be two minimal rational  del Pezzo surfaces  over $\mathbb{K}$ of degrees~$d$ and $d',$ respectively.   Denote by $n$ the minimal number of Sarkisov links composition of which gives a birational map between $X$ and $X'.$

\begin{enumerate}
\renewcommand\labelenumi{\rm (\roman{enumi})}
\renewcommand\theenumi{\rm (\roman{enumi})}

\item\label{lemma:links(5,9)}
Let $(d,d')=(5,9).$ Then $n=1.$

\item\label{lemma:links(8,9)}
Let $(d,d')=(8,9).$ Then $n=1.$

\item\label{lemma:links(5,8)}
Let $(d,d')=(5,8).$ Then $n=1.$

\item\label{lemma:links(5,5)}
Let $(d,d')=(5,5).$ Then $n \leqslant 2.$

\item\label{lemma:links(8,8)}
Let $(d,d')=(8,8).$ Then $n \leqslant 2.$

\item\label{lemma:links(6,9)}
Let $(d,d')=(6,9).$ Then $n=2.$

\item\label{lemma:links(5,6)}
Let $(d,d')=(5,6).$ Then $n=2.$

\item\label{lemma:links(6,8)}
Let $(d,d')=(6,8).$ Then $n\leqslant 2.$

\item\label{lemma:links(6,6)}
Let $(d,d')=(6,6).$ Then $n\leqslant 2.$

\end{enumerate}

\end{Lemma}

\begin{proof}

Note that by Theorem~\ref{th:SarkisovlinksdP} we get that if birational map $f$ between two minimal del Pezzo surfaces  is a composition of Sarkisov links $\chi_1, \chi_2, \ldots, \chi_n$ and at least one of $\chi_i$ is not of type~\rom{2}, then $n \geqslant 3.$ So in the proof we consider only Sarkisov links of type~\rom{2}.

Assertion~\ref{lemma:links(5,9)} immediately follows from Lemma~\ref{lemma:linkd9d5} and assertion~\ref{lemma:links(8,9)} immediately follows from Lemma~\ref{lemma:linkd9d8}.

Let us prove assertion~\ref{lemma:links(5,8)}. From Lemma~\ref{lemma:linkd8d5} we get a Sarkisov link of type~\rom{2} between any minimal del Pezzo surface $X$ of degree~$5$ and any minimal del Pezzo surface~$X'$ of degree~$8.$

Let us prove assertions~\ref{lemma:links(5,5)} and~\ref{lemma:links(8,8)}.   By Lemmas~\ref{lemma:linkd9d5} and~\ref{lemma:linkd9d8} there are two Sarkisov links of type~\rom{2}  
$$
\chi_1 \colon X  \DashedArrow[->,densely dashed    ]  \mathbb{P}^2 \;\; \text{and} \;\; \chi_2  \colon \mathbb{P}^2 \DashedArrow[->,densely dashed    ] X'.
$$
 So the composition~$\chi_2 \circ \chi_1$ gives the desired estimates.

Let us prove assertion~\ref{lemma:links(6,9)}. By Theorem~\ref{th:SarkisovlinksdP} there are no Sarkisov links between minimal del Pezzo surface $X$ of degree $6$  and $\mathbb{P}^2.$ However, by Lemmas~\ref{lemma:linkd8d6} and~\ref{lemma:linkd9d8} we get two Sarkisov links of type~\rom{2}
$$
\chi_1 \colon X  \DashedArrow[->,densely dashed    ]  X''  \;\; \text{and} \;\; \chi_2  \colon X'' \DashedArrow[->,densely dashed    ] \mathbb{P}^2,
$$
where $X''$ is a minimal rational del Pezzo surface of degree~$8.$ So the composition~\mbox{$\chi_2 \circ \chi_1$}  gives us $n=2.$

Let us prove assertion~\ref{lemma:links(5,6)}. By Theorem~\ref{th:SarkisovlinksdP} there are no Sarkisov links between minimal del Pezzo surfaces $X$ and $X'$ of degree $5$ and $6,$ respectively.  By Lemma~\ref{lemma:linkd8d6} there is a Sarkisov link of type~\rom{2} between minimal del Pezzo surface $X'$ of degree~$6$  and some minimal del Pezzo surface of degree~$8.$ Let $X''$ be such a  minimal del Pezzo surface of degree~$8$ and 
$$
\chi_1 \colon X'' \DashedArrow[->,densely dashed    ] X'
$$
be a Sarkisov link of type~\rom{2} between $X''$ and $X'.$ By Lemma~\ref{lemma:linkd8d5}  there is a Sarkisov link of type~\rom{2}  between minimal del Pezzo surface $X$ of degree $5$  and any minimal del Pezzo surface of degree~$8.$ Let
$$
\chi_2 \colon X \DashedArrow[->,densely dashed    ] X''
$$
be a Sarkisov link of type~\rom{2} between $X$ and  $X''.$  So the composition~\mbox{$\chi_1 \circ \chi_2$}  gives us $n=2.$

Let us prove assertion~\ref{lemma:links(6,8)}. By Lemma~\ref{lemma:dp6KL} the minimal rational  del Pezzo surface~$X$ of degree $6$  is defined by two fields $(\mathbf{K}, \mathbf{L})$ from Severi--Brauer data.   By Lemma~\ref{lemma:dP8Weil} we have $X' \simeq \mathrm{R}_{\mathbf{K}'/\mathbb{K}}(\mathbb{P}^1),$ where $\mathbb{K} \subset \mathbf{K}'$ is a quadratic field extension which is uniquely defined by $X'.$ By Lemma~\ref{lemma:linkd6d6p2}  there is a Sarkisov link of type~\rom{2}
$$
\chi_1 \colon X  \DashedArrow[->,densely dashed    ] X''
$$ 
between $X$ and a minimal rational del Pezzo surface $X''$ of degree~$6$ which is  defined  by the pair of fields $(\mathbf{K}',\mathbf{L}).$  By Lemma~\ref{lemma:linkd8d6}  there is a Sarkisov link of type~\rom{2}
$$
\chi_2 \colon X''  \DashedArrow[->,densely dashed    ] X'
$$ 
between $X''$ and  $X'\simeq \mathrm{R}_{\mathbf{K}'/\mathbb{K}}(\mathbb{P}^1).$ So the composition~\mbox{$\chi_2 \circ \chi_1$} gives us $n \leqslant 2.$

Let us prove assertion~\ref{lemma:links(6,6)}. Let $X$ and $X'$ be two minimal rational del Pezzo surfaces of degree~$6.$ By Lemma~\ref{lemma:dp6KL} the minimal rational  del Pezzo surfaces~$X$ and~$X'$ of degree~$6$ up to isomorphism  are defined by pairs of fields $(\mathbf{K}, \mathbf{L})$ and~$(\mathbf{K}', \mathbf{L}')$ from Severi--Brauer data, respectively. Then by Lemma~\ref{lemma:linkd6d6p2}   we get that if~\mbox{$\mathbf{K} \neq \mathbf{K}'$} to change the field $\mathbf{K}$ to $\mathbf{K}'$ one needs one Sarkisov link and by Lemma~\ref{lemma:linkd6d6p3} if $\mathbf{L} \neq \mathbf{L}'$  to change the field $\mathbf{L}$ to~$\mathbf{L}'$ one needs another Sarkisov link. So we get $n \leqslant 2.$

\end{proof}

\section{Proof of the main result}\label{section:proof}

In this section we prove Theorem~\ref{th:Sarkisovperfect}.

\begin{proof}[Proof of Theorem \ref{th:Sarkisovperfect}]
Let $n \in \mathbb{Z}_{\geqslant 0}$ be a minimal number with the property that there is a birational map between del Pezzo surfaces  $f  \colon X \DashedArrow[->,densely dashed    ] X'$  which is a composition of $n$ Sarkisov links. Denote by $d=K_X^2$ the degree of $X$ and by $d'=K^2_{X'}$ the degree of $X'.$  It is well known that
$$
1 \leqslant d \leqslant 9 \;\; \text{and} \;\; 1 \leqslant d' \leqslant 9.
$$
\noindent Also note that $d$ and $d'$ are different from~$7,$ because del Pezzo surface of degree~$7$ is never minimal.

If $1 \leqslant d \leqslant 3,$ then by~\cite[Theorems 1.6, 4.4 and 4.5]{Isk} we get that $X \simeq X'.$   If~\mbox{$d=4,$} then by~\cite[Theorems 1.1 and 1.2]{ShTr} (cf.~\cite[Theorem 1.14]{Shinder}) we get that  $X \simeq X'.$  This means that in the case $1 \leqslant d \leqslant 4,$ we obtain~\mbox{$n=0.$}

Therefore, we get that if $d \geqslant 5,$ then $d' \geqslant 5.$

Assume that $X$ is not rational. Then by Theorems~\ref{th:dp5point} and~\ref{theorem:rationalityd>5}  we get that 
$$
d, d' \in \{6,8,9\}.
$$
\noindent If $d=6,$ then by Proposition~\ref{prop:yas} we have that~$X'$ is also minimal  del Pezzo surface of degree $6$  and $n \leqslant 1.$ If $d=8,$ then by~\cite[Theorem~1.6]{Trepalin} we get that~\mbox{$X \simeq X'$} and thus,~\mbox{$n=0.$} If $d=9,$  then by Lemma~\ref{lemma:SB} we get $n \leqslant 1.$

If $X$ is rational, then by Lemma~\ref{lemma:d>5n<2} we get $n \leqslant 2.$ 

The last part of the theorem follows from Corollary~\ref{corollary:dp>2}.

\end{proof}

%
%
%

\providecommand{\bysame}{\leavevmode\hbox to3em{\hrulefill}\thinspace}
\providecommand{\MR}{\relax\ifhmode\unskip\space\fi MR }
\providecommand{\MRhref}[2]{%
  \href{http://www.ams.org/mathscinet-getitem?mr=#1}{#2}
}
\providecommand{\href}[2]{#2}

\end{document}